\documentclass{amsart}
\usepackage[latin9]{inputenc}
\usepackage{verbatim}
\usepackage{mathrsfs}
\usepackage{amstext}
\usepackage{amsthm}
\usepackage{amssymb}
\usepackage{setspace}
\makeatletter
\numberwithin{equation}{section}
\numberwithin{figure}{section}
\theoremstyle{plain}
\newtheorem{thm}{\protect\theoremname}
  \theoremstyle{plain}
  \newtheorem{conjecture}[thm]{\protect\conjecturename}
  \theoremstyle{plain}
  \newtheorem{prop}[thm]{\protect\propositionname}
  \theoremstyle{remark}
  \newtheorem{rem}[thm]{\protect\remarkname}
  \theoremstyle{plain}
  \newtheorem{lem}[thm]{\protect\lemmaname}

\newcommand{\Affi}{\mathbb A}

{.tfm}
{.tfm}

  \providecommand{\propositionname}{Proposition}
\providecommand{\theoremname}{Theorem}

  \providecommand{\conjecturename}{Conjecture}
  
  \providecommand{\lemmaname}{Lemma}
  \providecommand{\propositionname}{Proposition}
  \providecommand{\remarkname}{Remark}
\providecommand{\theoremname}{Theorem}

\makeatother

  \providecommand{\conjecturename}{Conjecture}
  \providecommand{\lemmaname}{Lemma}
  \providecommand{\propositionname}{Proposition}
  \providecommand{\remarkname}{Remark}
\providecommand{\theoremname}{Theorem}

\begin{document}

\title[A Local-Global Equality on Every Affine Variety]{A Local-Global Equality on Every Affine Variety Admitting Points
in An Arbitrary Rank-One Subgroup of a Global Function Field}
\begin{abstract}
For every affine variety over a global function field, we show that
the set of its points with coordinates in an arbitrary rank-one multiplicative
subgroup of this function field is topologically dense in the set
of its points with coordinates in the topological closure of this
subgroup in the product of the multiplicative group of those local
completions of this function field over all but finitely many places. 
\end{abstract}

\author{Chia-Liang Sun}
\maketitle

\section{Introduction}

Let $K$ be a global function field over a finite field $k$ of positive
characteristic $p$. We denote by $k^{\text{alg}}$ the algebraic
closure of $k$ inside a fixed algebraic closure $K^{\text{alg}}$
of $K$. Let $\Sigma_{K}$ be the set of all places of $K$. For each
$v\in\Sigma_{K}$, denote by $K_{v}$ the completion of $K$ at $v$;
by $O_{v}$, $\mathfrak{m}_{v}$, and $\mathbb{F}_{v}$ respectively
the valuation ring, the maximal ideal, and the residue field associated
to $v$. For each finite subset $S\subset\Sigma_{K}$, we denote by
$O_{S}$ the ring of $S$-integers in $K$. For any commutative ring
$R$ with unity, denote by $R^{*}$ the group of its units. We fix
a subgroup $\Gamma\subset O_{S}^{*}$ for some finite $S\subset\Sigma_{K}$.
Let $M$ be a natural number, and $\Affi^{M}$ be the affine $M$-space,
whose coordinate is denoted by ${\mathbf{X}}=(X_{1},\ldots,X_{M})$.
For each polynomial $f\in K[X_{1},\ldots,X_{M}]$, we denote by $H_{f}$
the hypersurface in $\Affi^{M}$ defined by $f$. If the total degree
of $f$ is one, we say that $H_{f}$ is a hyperplane. By a linear
$K$-variety in $\Affi^{M}$, we mean an intersection of $K$-hyperplanes.
We say that a closed $K$-variety $W$ in $\Affi^{M}$ is homogeneous
if $W$ can be defined by homogeneous polynomials.

For any closed $K$-variety $W$ in $\Affi^{M}$ and any subset $\Theta$
of some ring containing $K$, let $W(\Theta)$ denote the set of points
on $W$ with each coordinate in $\Theta$. For each subset $\widetilde{S}\subset\Sigma_{K}$,
we endow $\prod_{v\in\widetilde{S}}K_{v}^{*}$ with the natural product
topology; via the diagonal embedding, we identify $W(\Gamma)$ with
its image $W(\Gamma)_{\widetilde{S}}$ in $W\left(\prod_{v\in\widetilde{S}}K_{v}^{*}\right)$
and denote by $\overline{W(\Gamma)_{\widetilde{S}}}$ its topological
closure. We naturally identify $\Gamma$ with $\Affi^{1}(\Gamma)$,
and write $\overline{\Gamma_{\widetilde{S}}}$ for $\overline{\Affi^{1}(\Gamma)_{\widetilde{S}}}$.
For each place $v\in\Sigma_{K}$, we write $\overline{\Gamma_{v}}$
for $\overline{\Gamma_{\{v\}}}$. Note that $\overline{\Gamma_{\widetilde{S}}}\subset\prod_{v\in\widetilde{S}}\overline{\Gamma_{v}}$.
We fix a cofinite subset $\Sigma\subset\Sigma_{K}$, and drop the
lower subscript $\Sigma$ in the notation of topological closure;
for example, we simply write $\overline{\Gamma}$ for $\overline{\Gamma_{\Sigma}}$.

For any closed $K$-variety $W$ in $\Affi^{M}$, we consider the
following conjecture. 
\begin{conjecture}
\label{conj: my} $W(\overline{\Gamma})=\overline{W(\Gamma)}$.
\end{conjecture}
First formulated in this form by the author \cite{Sun-Skolem_Conj},
Conjecture \ref{conj: my} is an analog for split algebraic tori to
Conjecture C in \cite{PoonenVoloch} for Abelian varieties. One of
the deepest aspects of Conjecture \ref{conj: my} is explained as
follows. If $W=\bigcup_{i\in I}W_{i}$ is a finite union of closed
$K$-varieties in $\mathbb{A}^{M}$, then it is easy to see that 
\[
W(\overline{\Gamma})\supset\bigcup_{i\in I}W_{i}(\overline{\Gamma})\supset\bigcup_{i\in I}\overline{W_{i}(\Gamma)}=\overline{W(\Gamma)};
\]
thus if Conjecture \ref{conj: my} holds for $W$, then we must have
$W(\overline{\Gamma})=\bigcup_{i\in I}W_{i}(\overline{\Gamma})$.
In fact, following the idea proposed by Stoll (Question 3.12 in \cite{Stoll--Finite_descent_obstructions},
so-called ``Adelic Mordell-Lang Conjecture\textquotedblright ) and
first realized by Poonen and Voloch \cite{PoonenVoloch}, all previous
results \cite{Sun,Sun-Skolem_Conj} dealing with Conjecture \ref{conj: my}
for reducible $W$ are established by reducing it to the assertion
that $Z(\overline{\Gamma})=\bigcup_{i\in I}Z_{i}(\overline{\Gamma})$
for every finite union $Z=\bigcup_{i\in I}Z_{i}$ of irreducible zero-dimensional
$K$-varieties in $\mathbb{A}^{M}$, and proving this assertion via
an argument invented by Poonen and Voloch \cite{PoonenVoloch}, who
managed to bypass the difficulty encountered when one tries to develop
the function-field analog of the proof by Stoll \cite{Stoll--Finite_descent_obstructions}
of the number-field counterpart of this assertion. In the present
setting, the Mordell-Lang Conjecture is treated by Derksen and Masser
\cite{DM} in full generalities. The author \cite{Sun-Skolem_Conj}
establishes some ``adelic analog'' (restated as Proposition \ref{prop:old_red}
below) of their result in certain case, and completes the aforementioned
reduction under the artificial hypothesis induced from this analog;
this hypothesis is then put in the main result in \cite{Sun-Skolem_Conj}
on Conjecture \ref{conj: my}. In the remaining case, however, no
adelic analog exists (see Remark \ref{rem:old_red}); thus it is hopeless
to completely solve Conjecture \ref{conj: my} via this ``Mordell-Lang
approach''.

In this paper, we tackle Conjecture \ref{conj: my} in a different
approach; although we still need Proposition \ref{prop:old_red},
its present usage is to reduce Conjecture \ref{conj: my} to the situation
where $W$ is defined over a hopefully smaller subfield of $K$ without
putting on any assumption on $W$. In the case where $\Gamma$ has
rank one, we directly prove Conjecture \ref{conj: my} in this situation
by generalizing of Proposition 24 in the author's recent work \cite{Sun-const_hyper}
via introducing Lemma \ref{lem:lin_alg}, which is an elementary linear-algebra
argument. Our approach leads to the following main result, in which
no hypothesis is put on $W$.
\begin{thm}
\label{thm:main}Suppose that $\Gamma\cap O_{S}^{*}$ has rank at
most one, where $S=\Sigma_{K}\setminus\Sigma$. Then for every closed
$K$-variety $W$ in $\Affi^{M}$, we have that $W(\overline{\Gamma})=\overline{W(\Gamma)}$.
\end{thm}

\section{The proof of Theorem \ref{thm:main}\label{sec:The-proofs}}

For any subgroup $\Delta\subset K^{*}$, we denote by $k(\Delta)$
the smallest subfield of $K$ containing $k$ and $\Delta$, by $\rho(\Delta)$
the subgroup $\bigcap_{m\geq0}(K^{p^{m}})^{*}\Delta$ of $K^{*}$,
and by $\sqrt[^{K}]{\Delta}$ the subgroup $\{x\in K^{*}:x^{n}\in\Delta\text{ for some }n\in\mathbb{N}\}$
of $K^{*}$. The following result is a special case of Proposition
6 in \cite{Sun-Skolem_Conj}.
\begin{prop}
\label{prop:old_red}Let $d$ be the dimension of $W$. Suppose that
$W$ is a union of homogeneous linear $K$-varieties, and that each
$d$-dimensional irreducible component of $W$ is not $\rho(\Gamma)$-isotrivial.
Then there exists a finite union $V$ of homogeneous linear $K$-subvarieties
of $W$ with dimension smaller than $d$ such that $W(\overline{\Gamma_{v}})=V(\overline{\Gamma_{v}})$
for every $v\in\Sigma_{K}$; in particular, we have $W(\overline{\Gamma})=V(\overline{\Gamma})$. 
\end{prop}
\begin{rem}
\label{rem:old_red}Proposition \ref{prop:old_red} is an analog for
$W(\overline{\Gamma})$ to the qualitative part of conclusion (a)
in Main Estimate of Section 9 in \cite{DM}. However, there is no
analog for $W(\overline{\Gamma})$ to the qualitative part of conclusion
(b). To be precise, we consider the example where $W=H_{f}\subset\Affi^{2}$
with $f(X_{1},X_{2})=X_{1}+X_{2}-1\in k[X_{1},X_{2}]$. Note that
$W$ is irreducible of dimension one. By the qualitative part of conclusion
(b), there is a finite set $\mathcal{V}$ of irreducible proper $K$-subvarieties
of $W$ such that $W\left(\sqrt[^{K}]{\Gamma}\right)=\bigcup_{V\in\mathcal{V}}\bigcup_{e\in\mathbb{N}\cup\{0\}}\left(V\left(\sqrt[^{K}]{\Gamma}\right)\right)^{p^{e}}$.
Considering the proper $K$-subvariety $Z_{0}=\bigcup_{V\in\mathcal{V}}V$
of $W$, we see that $\bigcup_{V\in\mathcal{V}}\left(V\left(\sqrt[^{K}]{\Gamma}\right)\right)^{p^{e}}=\left(Z\left(\sqrt[^{K}]{\Gamma}\right)\right)^{p^{e}}$
for each $e\in\mathbb{N}\cup\{0\}$; it follows that 
\[
W\left(\sqrt[^{K}]{\Gamma}\right)=\bigcup_{e\in\mathbb{N}\cup\{0\}}\left(Z\left(\sqrt[^{K}]{\Gamma}\right)\right)^{p^{e}}.
\]
Nevertheless, in the case where $K=\mathbb{F}_{p}(T)$ and $\Gamma=\{cT^{n}(1-T)^{m}:\,c\in\mathbb{F}_{p}^{*},\,(n,m)\in\mathbb{Z}^{2}\}=\sqrt[^{K}]{\Gamma}$
and $\Sigma$ is the maximal subset of $\Sigma_{K}$ such that $\Gamma\subset O_{v}^{*}$
for each $v\in\Sigma$, we claim that 
\[
W\left(\overline{\Gamma}\right)\not\subset\bigcup_{e\in\mathbb{N}\cup\{0\}}\left(Z\left(\overline{\Gamma}\right)\right)^{p^{e}}
\]
for every proper $K$-subvariety $Z$ of $W$. To see this, first
note that such $Z$ must have dimension zero. Moreover, the sequence
$(T^{p^{n!}})_{n\ge0}$ converges to some element $\alpha=(\alpha_{v})_{v\in\Sigma}\in\overline{\Gamma_{T}}\subset\overline{\Gamma}$,
where $\Gamma_{T}\subset\Gamma$ is the cyclic subgroup generated
by $T$. (Example 1 in \cite{Sun}) Similarly, we see that $1-\alpha\in\overline{\Gamma_{1-T}}\subset\overline{\Gamma}$,
where $\Gamma_{1-T}\subset\Gamma$ is the cyclic subgroup generated
by $T$. Thus we have that $(\alpha,1-\alpha)\in W(\overline{\Gamma})$.
Based on these facts, our claim can be proved by either of the two
following arguments.
\begin{enumerate}
\item Example 1 in \cite{Sun} also shows that $\alpha\notin K^{*}$, thus
$(\alpha,1-\alpha)\not\in W(K^{*})$. However, since $Z$ has dimension
zero, we have that $Z\left(\overline{\Gamma}\right)=Z\left(\Gamma\right)$.
(Proposition 2 in \cite{Sun-Skolem_Conj}) Because $Z\left(\Gamma\right)\subset Z(K^{*})\subset W(K^{*})$,
this proves our claim.
\item For each $v\in\Sigma$, note that $\alpha_{v}\in O_{v}^{*}$ and let
$P_{v}(T)\in\mathbb{F}_{p}[T]$ be the unique irreducible polynomial
such that $P_{v}(T)\in\mathfrak{m}_{v}$; thus we have that $P_{v}(T^{p^{n!}})=P_{v}(T)^{p^{n!}}\in\mathfrak{m}_{v}^{p^{n!}}$
for each $n\ge0$, which implies that $P_{v}(\alpha_{v})\in\cap_{n\ge0}\mathfrak{m}_{v}^{p^{n!}}=\{0\}$;
it follows that $P_{v}$ is the minimal polynomial for $\alpha_{v}$
over $\mathbb{F}_{p}$. On the other hand, for any $\alpha\in k^{\text{alg}}$
and any $e\in\mathbb{N}\cup\{0\}$, the element $\alpha^{p^{e}}$
is a zero of the minimal point for $\alpha$ over $\mathbb{F}_{P}$.
As $Z$ is a zero-dimensional $K$-variety, it follows that the degrees
of minimal polynomials of torsion points in $\left(Z\left(\overline{\Gamma_{v}}\right)\right)^{p^{e}}$
are uniformly bounded over all $(v,e)\in\Sigma\times(\mathbb{N}\cup\{0\}$).
Because the degree of $P_{v}$ can be arbitrarily large as $v$ ranges
over $\Sigma$, there must be some $v_{0}\in\Sigma$ such that $\alpha_{v_{0}}\notin\bigcup_{e\in\mathbb{N}\cup\{0\}}\left(Z\left(\overline{\Gamma_{v_{0}}}\right)\right)^{p^{e}}$;
since $Z\left(\overline{\Gamma}\right)\subset\prod_{v\in\Sigma}Z\left(\overline{\Gamma_{v}}\right)$,
this proves our claim.
\end{enumerate}
\end{rem}
\begin{prop}
\label{cor:old_red}For any closed $K$-variety $W\subset\Affi^{M}$,
there exists some closed $k(\rho(\Gamma))$-subvariety $V$ of $W$
such that $W(\overline{\Gamma_{v}})=V(\overline{\Gamma_{v}})$ for
every $v\in\Sigma_{K}$; in particular, we have $W(\overline{\Gamma})=V(\overline{\Gamma})$. 
\end{prop}
\begin{proof}
Let $\{f_{j}:\,1\le j\le J\}\subset K[X_{1},\ldots,X_{M}]$ be a set
of polynomials defining $W$. Choose $D\in\mathbb{N}$ such that for
each $j\in\{1,\ldots,J\}$, we may write 
\[
f_{j}(X_{1},\ldots,X_{M})=\sum_{(d_{1},\cdots,d_{M})\in\{0,1,\cdots,D\}^{M}}c_{(j,d_{1},\cdots,d_{M})}X_{1}^{d_{1}}\cdots X_{M}^{d_{M}}
\]
with each $c_{(j,d_{1},\cdots,d_{M})}\in K$. Consider the tuple $\mathbf{Y}=(Y_{(d_{1},\cdots,d_{M})})_{(d_{1},\cdots,d_{M})\in\{0,1,\cdots,D\}^{M}}$
of new variables, in which we define linear forms 
\[
\ell_{j}(\mathbf{Y})=\sum_{(d_{1},\cdots,d_{M})\in\{0,1,\cdots,D\}^{M}}c_{(j,d_{1},\cdots,d_{M})}Y_{(d_{1},\cdots,d_{M})}
\]
for each $j\in\{1,\ldots,J\}$. Let $N=(D+1)^{M}$ and $W'\subset\mathbb{A}^{N}$
be the homogeneous linear variety defined by $\{\ell_{j}:\,1\le j\le J\}$.
By Proposition \ref{prop:old_red}, there exists a finite union $V'$
of homogeneous linear $K$-subvarieties of $W'$ such that each irreducible
component of $V'$ is $\rho(\Gamma)$-isotrivial and that $W'(\overline{\Gamma_{v}})=V'(\overline{\Gamma_{v}})$
for every $v\in\Sigma_{K}$. In particular, each irreducible component
of $V'$ is defined over $k(\rho(\Gamma))$, thus so is $V'$. Let
$\{g_{j}':\,1\le j\le J'\}\subset k(\rho(\Gamma))[\mathbf{Y}]$ be
a set of polynomials defining $V'$. For each $j\in\{1,\ldots,J'\}$,
we construct $f'_{j}(X_{1},\ldots,X_{M})$ by substituting each variable
$Y_{(d_{1},\cdots,d_{M})}$ in $g_{j}'(\mathbf{Y})$ by the monomial
$X_{1}^{d_{1}}\cdots X_{M}^{d_{M}}$, thus we have that $f'_{j}(X_{1},\ldots,X_{M})\in k(\rho(\Gamma))[X_{1},\ldots,X_{M}]$.
Let $V\subset\mathbb{A}^{M}$ be the $k(\rho(\Gamma))$-variety whose
vanishing ideal is generated defined by $\{f_{j}':\,1\le j\le J'\}$.
For every $j\in\{1,\ldots,J'\}$ and every $(x_{1},\ldots,x_{M})\in V(K^{\text{alg}})$,
we have $f'_{j}(x_{1},\ldots,x_{M})=0$, thus the point $(x_{1}^{d_{1}}\cdots x_{M}^{d_{M}})_{(d_{1},\cdots,d_{M})\in\{0,1,\cdots,D\}^{M}}\in\mathbb{A}^{N}(K^{\text{alg}})$
is a zero of $g_{j}'(\mathbf{Y})$ by construction; this shows that
$(x_{1}^{d_{1}}\cdots x_{M}^{d_{M}})_{(d_{1},\cdots,d_{M})\in\{0,1,\cdots,D\}^{M}}\in V'(K^{\text{alg}})\subset W'(K^{\text{alg}})$
and thus the construction yields $(x_{1},\ldots,x_{M})\in W(K^{\text{alg}})$.
Hence we see that $V\subset W$. Similar reasonings gives the other
desired conclusion that $W(\overline{\Gamma_{v}})=V(\overline{\Gamma_{v}})$
for every $v\in\Sigma_{K}$.
\end{proof}
For any finitely generated subgroup $\Delta\subset K^{*}$, Lemma
3 of \cite{axby} shows that $\rho(\Delta)\subset\sqrt[^{K}]{\Delta}$,
and thus that $\Delta$ and $\rho(\Delta)$ have the same rank; we
also note that $\Delta\subset\rho(\Delta)=\rho(\rho(\Delta))$ by
definition.
\begin{prop}
\label{prop:red}Letting $S=\Sigma_{K}\setminus\Sigma$, there exists
a free subgroup $\Phi\subset O_{S}^{*}$ which has the same rank as
$\Gamma\cap O_{S}^{*}$ and satisfies the following property: if $V(\overline{\Phi})=\overline{V(\Phi)}$
for every closed $k(\Phi)$-variety $V\subset\Affi^{M}$, then $W(\overline{\Gamma})=\overline{W(\Gamma)}$
for every closed $K$-variety $W\subset\Affi^{M}$.
\end{prop}
\begin{proof}
Let $\Phi$ be a maximal free subgroup of the finitely generated abelian
group $\rho(\Gamma\cap O_{S}^{*})$. Since $\Phi\subset\rho(\Phi)\subset\rho(\rho(\Gamma\cap O_{S}^{*}))=\rho(\Gamma\cap O_{S}^{*})$,
it follows that $\Phi$ is a maximal free subgroup of $\rho(\Phi)$,
and this implies that $\rho(\Phi)=\text{Tor}(\rho(\Phi))\Phi=k^{*}\Phi$.
Letting $S_{0}\subset\Sigma_{K}$ be a finite subset such that $\Gamma\subset O_{S_{0}}^{*}$,
we see that the image of $\Gamma$ in $\prod_{v\in\Sigma}K_{v}^{*}$
is contained in $\left(\prod_{v\in\Sigma\cap S_{0}}K_{v}^{*}\right)\times\left(\prod_{v\in\Sigma\setminus S_{0}}O_{v}^{*}\right)$;
since $S=\Sigma_{K}\setminus\Sigma$, this shows that the image of
$\Gamma\cap O_{S}^{*}$ in $\prod_{v\in\Sigma}K_{v}^{*}$ is exactly
the intersection of the image of $\Gamma$ in $\prod_{v\in\Sigma}K_{v}^{*}$
with the open subgroup $\left(\prod_{v\in\Sigma\cap S_{0}}O_{v}^{*}\right)\times\left(\prod_{v\in\Sigma\setminus S_{0}}K_{v}^{*}\right)$
of $\prod_{v\in\Sigma}K_{v}^{*}$. It follows that $\Gamma\cap O_{S}^{*}$
is open in $\Gamma$. Since the index of $\Phi\cap\Gamma\cap O_{S}^{*}$
in $\Gamma\cap O_{S}^{*}$ is finite, Corollary 2 of \cite{Sun-Skolem_Conj}
shows that $\Phi\cap\Gamma\cap O_{S}^{*}$ is open in $\Gamma\cap O_{S}^{*}$,
and thus is open in $\Gamma$. We note that $\Phi\cap\Gamma\cap O_{S}^{*}=\Phi\cap\Gamma$
since $\Phi\subset O_{S}^{*}$.

Fix a closed $K$-variety $W\subset\Affi^{M}$. Consider an arbitrary
$\mathbf{x}\in W(\overline{\Gamma})$, which is the limit of a sequence
$\left(\mathbf{x}_{n}\right)_{n\in\mathbb{N}}$ in $\mathbb{A}^{M}(\Gamma)$.
Since $\Phi\cap\Gamma$ is open in $\Gamma$, we may assume that $\mathbf{x}_{n}=\mathbf{r}\mathbf{y}_{n}$
with some $\mathbf{r}\in\mathbb{A}^{M}(\Gamma)$ and a sequence $\left(\mathbf{y}_{n}\right)_{n\in\mathbb{N}}$
in $\mathbb{A}^{M}(\Phi\cap\Gamma)$. Note that the sequence $\left(\mathbf{y}_{n}\right)_{n\in\mathbb{N}}$
converges to $\mathbf{r}^{-1}\mathbf{x}\in(\mathbf{r}^{-1}W)(\overline{\Phi})$.
Recalling that $\rho(\Phi)=k^{*}\Phi$, Proposition \ref{cor:old_red}
says that there exist some closed $k(\Phi)$-subvariety $V$ of $\mathbf{r}^{-1}W$
such that $(\mathbf{r}^{-1}W)(\overline{\Phi})=V(\overline{\Phi})$.
Assuming $V(\overline{\Phi})=\overline{V(\Phi)}$, we see that $\mathbf{r}^{-1}\mathbf{x}\in(\mathbf{r}^{-1}W)(\overline{\Phi})=V(\overline{\Phi})=\overline{V(\Phi)}\subset\overline{(\mathbf{r}^{-1}W)(\Phi)}\subset\mathbf{r}^{-1}\left(\overline{W(\Phi)}\right)$,
i.e. $\mathbf{x}\in\overline{W(\Phi)}$ is the limit of some sequence
$\left(\mathbf{x}'_{n}\right)_{n\in\mathbb{N}}$ in $W(\Phi)$. Letting
$\left(\mathbf{x}''_{n}\right)_{n\in\mathbb{N}}\subset\mathbb{A}^{M}(\Phi\Gamma)$
be the sequence defined by $\mathbf{x}''_{2n-1}=\mathbf{x}_{n}$ and
$\mathbf{x}''_{2n}=\mathbf{x}'_{n}$, we see that the sequence $\left(\mathbf{x}''_{n}\right)_{n\in\mathbb{N}}\subset\mathbb{A}^{M}(\Phi\Gamma)$
is Cauchy. As abelian groups, $\Phi\Gamma/\Gamma$ is isomorphic to
$\Phi/(\Gamma\cap\Phi)$, which is finite by the construction of $\Phi$;
thus Corollary 2 of \cite{Sun-Skolem_Conj} shows that $\Gamma$ is
open in $\Phi\Gamma$. It follows that $\Phi\cap\Gamma$ is open in
$\Phi\Gamma$. Hence, for sufficiently large $n\in\mathbb{N}$, we
have that $(\mathbf{r}^{-1}\mathbf{x}_{n})^{-1}(\mathbf{r}^{-1}\mathbf{x}'_{n})=(\mathbf{x}''_{2n-1})^{-1}\mathbf{x}''_{2n}\in\mathbb{A}^{M}(\Phi\cap\Gamma)$;
since $\mathbf{r}^{-1}\mathbf{x}_{n}=\mathbf{y}_{n}\in\mathbb{A}^{M}(\Phi\cap\Gamma)$,
we conclude that $\mathbf{r}^{-1}\mathbf{x}'_{n}\in\mathbb{A}^{M}(\Phi\cap\Gamma)$
and thus $\mathbf{x}'_{n}\in\mathbb{A}^{M}(\Gamma)\cap W(\Phi)\subset W(\Gamma)$,
i.e. $\mathbf{x}\in\overline{W(\Gamma)}$. This finishes the proof.

\end{proof}
For any $a\in\mathbb{N}$ and $b\in\mathbb{N}\setminus p\mathbb{N}$,
consider the polynomial 
\[
g_{a,b}(T)=\frac{T^{ab}-1}{T^{a}-1}\in k[T].
\]

We make the following convention. For a polynomial $Q(T)\in k[T]$
and a rational function $P(T)\in k(T)$, we say that $Q(T)$ divides
$P(T)$ if any zero of $Q(T)$ in $k^{\text{alg}}$ is not a pole
of $\frac{P(T)}{Q(T)}$. The long proof of the following proposition,
which is the core in the proof of Theorem \ref{thm:main}, is postponed
to Section \ref{sec:The-Proof-of-key}. 
\begin{prop}
\label{key} Let $\mathcal{S}$ be a finite set of irreducible polynomials
in $k[T]$. Let $J$ be a natural number. For each $j\in\{1,\ldots,J\}$,
let $f_{j}(X_{1},\ldots,X_{M})\in k[T][X_{1},\ldots,X_{M}]$.

Assume that there exists a sequence $\{(e_{1,n},\ldots,e_{M,n})\}_{n\ge1}$
in $\mathbb{A}^{M}(\mathbb{Z})$ satisfying the following conditions:

For every $Q(T)\in k[T]$ not divisible by any element in $\mathcal{S}$,
there is an $N_{Q}\in\mathbb{N}$ such that for any $n\ge N_{Q}$
we have that $Q(T)$ divides $f_{j}(T^{e_{1,n}},\ldots,T^{e_{M,n}})$
for all $j\in\{1,\ldots,J\}$. 

Then there exists a sequence $\{(e'_{1,n},\ldots,e'_{M,n})\}_{n\in\mathcal{N}}$
in $\mathbb{A}^{M}(\mathbb{Z})$ indexed by an infinite subset $\mathcal{N}\subset\mathbb{N}$
with the following properties:

\begin{enumerate}
\item For each $n\in\mathcal{N}$ we have $f_{j}(T^{e'_{1,n}},\ldots,T^{e'_{M,n}})=0$
for all $j\in\{1,\ldots,J\}$.\label{enu:p1}
\item For every $\widetilde{Q}(T)\in k[T]$ not divisible by $T$, there
is an $\widetilde{N}_{\widetilde{Q}}\in\mathbb{N}$ such that for
any $n\in\mathcal{N}$ with $n\ge\widetilde{N}_{\widetilde{Q}}$ we
have that $\widetilde{Q}(T)$ divides $T^{e_{i,n}}-T^{e'_{i,n}}$
for all $i$.\label{enu:p2}
\end{enumerate}
\end{prop}
The next theorem, which follows formally from Proposition \ref{key},
is proved in the same way that in \cite{Sun-const_hyper} Theorem
25 is formally deduced from Proposition 24.
\begin{thm}
\label{adelic}Let $W$ be a closed $k(\Gamma)$-variety in $\Affi^{M}$.
Suppose that $\Gamma$ is free with rank one, and that is contained
in $O_{S}^{*}$, where $S=\Sigma_{K}\setminus\Sigma$. Then we have
that $W(\overline{\Gamma})=\overline{W(\Gamma)}$. 
\end{thm}
\begin{proof}
Let $\gamma$ be a generator of $\Gamma$. Let $\Sigma|_{k(\gamma)}\subset\Sigma$
be the subset satisfying the following property that for each $v\in\Sigma$
there exists a unique $w\in\Sigma|_{k(\gamma)}$ such that both $v$
and $w$ restrict to the same place of $k(\gamma)$. Consider the
$k$-isomorphism between fields 
\begin{equation}
k(T)\rightarrow k(\gamma),\,\,T\mapsto\gamma.\label{eq:iso}
\end{equation}
Through the isomorphism (\ref{eq:iso}), the set $\Sigma|_{k(\gamma)}$
is injectively mapped onto a subset of the set of places of $k(T)$.
For each $v\in\Sigma|_{k(\gamma)}$, we have that $\gamma\in O_{v}^{*}$;
let $P_{v}(T)\in k[T]$ be the irreducible polynomial corresponding
to the image of $v$ under this map. Let $\mathcal{S}$ be the complement
of the subset $\{P_{v}(T):\,v\in\Sigma|_{k(\gamma)}\}$ of the set
of all irreducible polynomials in $k[T]$. Note that $\mathcal{S}$
is a finite set containing the polynomial $T$, and that $k[\Gamma]\subset\prod_{v\in\Sigma}O_{v}$,
where $k[\Gamma]$ is the smallest subring of $K$ containing both
$k$ and $\Gamma$. 

Write $W=\bigcap_{j=1}^{J}H_{f_{j}}$, where $f_{j}(X_{1},\ldots,X_{M})\in k[\gamma][X_{1},\ldots,X_{M}]$
for each $j$. Let $\{(\gamma^{e_{1,n}},\ldots,\gamma^{e_{M,n}})\}_{n\ge1}$
be a sequence in $\mathbb{A}^{M}(\Gamma)$ which converges to a point
$(x_{1},\ldots,x_{M})\in W(\overline{\Gamma})\subset\mathbb{A}^{M}\left(\prod_{v\in\Sigma}K_{v}^{*}\right)$,
where $e_{i,n}\in\mathbb{Z}$. In fact, this sequence lies in the
image of $\mathbb{A}^{M}\left(\prod_{v\in\Sigma|_{k(\gamma)}}k(\gamma)_{v}^{*}\right)$
in $\mathbb{A}^{M}\left(\prod_{v\in\Sigma}K_{v}^{*}\right)$ under
the natural map, where $k(\gamma)_{v}$ denotes the topological closure
of the subfield $k(\gamma)$ in $K_{v}$. Note that this image is
a closed subset. The topology on $\overline{\Gamma}$ is induced from
the usual product topology on $\prod_{v\in\Sigma}k(\gamma)_{v}^{*}$,
and the latter topology is the same as the subspace topology restricted
from the usual product topology on $\prod_{v\in\Sigma}k(\gamma)_{v}$.
Thus for each $i\in\{1,\ldots,M\}$ the sequence $\left(\gamma^{e_{i,n}}\right)_{n\ge1}$
converges to $x_{i}$ in $\prod_{v\in\Sigma}k(\gamma)_{v}$ Therefore,
from the continuity of each $f_{j}$ at $(x_{1},\ldots,x_{M})\in\mathbb{A}^{M}\left(\prod_{v\in\Sigma}k(\gamma)_{v}\right)$,
we see that each sequence $\left(f_{j}(\gamma^{e_{1,n}},\ldots,\gamma^{e_{M,n}})\right)_{n\ge1}$
converges to $f_{j}(x_{1},\ldots,x_{M})=0$ in $\prod_{v\in\Sigma}k(\gamma)_{v}$.
Consider the sequence $\{(e_{1,n},\ldots,e_{M,n})\}_{n\ge1}$ in $\mathbb{A}^{M}(\mathbb{Z})$.
Fix an arbitrary $Q(T)\in k[T]$ not divisible by any element in $\mathcal{S}$.
Thus we have the prime decomposition $Q(T)=\prod_{v\in\Sigma|_{k(\gamma)}}P_{v}(T)^{n_{v}}$
in $k[T]$, where there are only finitely many $v\in\Sigma|_{k(\gamma)}$
with $n_{v}>0$. In particular, 
\[
U_{Q}=\prod_{\begin{array}{c}
v\in\Sigma|_{k(\gamma)}\\
n_{v}=0
\end{array}}k(\gamma)_{v}\times\prod_{\begin{array}{c}
v\in\Sigma|_{k(\gamma)}\\
n_{v}>0
\end{array}}(\mathfrak{m}_{v}\cap k(\gamma)_{v})^{n_{v}}
\]
is an open subset in $\prod_{v\in\Sigma|_{k(\gamma)}}k(\gamma)_{v}$
endowed with the the product topology. Note that $f_{j}(\gamma^{e_{1,n}},\ldots,\gamma^{e_{M,n}})\in k[\gamma,\gamma^{-1}]$
for each $j\in\{1,\ldots,J\}$ and $n\in\mathbb{N}$. The intersection
of $U_{Q}$ with the image of $k[\gamma,\gamma^{-1}]$ in $\prod_{v\in\Sigma|_{k(\gamma)}}k(\gamma)_{v}$
is the image of $Q[\gamma]k[\gamma,\gamma^{-1}]$, which is thus an
open subset of $k[\gamma,\gamma^{-1}]$ containing zero with respect
to the subspace topology restricted from $\prod_{v\in\Sigma|_{k(\gamma)}}k(\gamma)_{v}$.
Therefore, from the fact each sequence $\left(f_{j}(\gamma^{e_{1,n}},\ldots,\gamma^{e_{M,n}})\right)_{n\ge1}$
converges to zero in $\prod_{v\in\Sigma}k(\gamma)_{v}$, it follows
that there is an $N_{Q}\in\mathbb{N}$ such that for any $n\ge N_{Q}$
we have that $f_{j}(\gamma^{e_{1,n}},\ldots,\gamma^{e_{M,n}})\in Q[\gamma]k[\gamma,\gamma^{-1}]$
for each $j\in\{1,\ldots,J\}$; thus by the isomorphism (\ref{eq:iso})
we have that $Q(T)$ divides $f_{j}(T^{e_{1,n}},\ldots,T^{e_{M,n}})$,
because $0$ is not a zero of $Q(T)$. Therefore the assumption of
Proposition \ref{key} is verified. Applying the isomorphism (\ref{eq:iso})
to the conclusion of Proposition \ref{key}, we see that there exists
a sequence $\{(e'_{1,n},\ldots,e'_{M,n})\}_{n\in\mathcal{N}}$ in
$\mathbb{A}^{M}(\mathbb{Z})$ indexed by an infinite subset $\mathcal{N}\subset\mathbb{N}$
satisfying the following properties: for each $n\in\mathcal{N}$ we
have $f_{j}(\gamma^{e'_{1,n}},\ldots,\gamma^{e'_{M,n}})=0$ for all
$j\in\{1,\ldots,J\}$, and for every $\widetilde{Q}(T)\in k[T]$ not
divisible by $T$, there is an $\widetilde{N}_{\widetilde{Q}}\in\mathbb{N}$
such that for any $n\in\mathcal{N}$ with $n\ge\widetilde{N}_{\widetilde{Q}}$
we have that $\gamma^{e_{i,n}}-\gamma^{e'_{i,n}}\in\widetilde{Q}(\gamma)k[\gamma,\gamma^{-1}]$
for all $i\in\{1,\ldots,M\}$. The first property says that $(\gamma^{e'_{1,n}},\ldots,\gamma^{e'_{M,n}})\in W(\Gamma)$
for each $n\in\mathcal{N}$. On the other hand, because the image
of $k[\gamma,\gamma^{-1}]$ in $\prod_{v\in\Sigma|_{k(\gamma)}}k(\gamma)_{v}$
lies in $\prod_{v\in\Sigma|_{k(\gamma)}}(O_{v}\cap k(\gamma)_{v})$,
one may argue similarly as above that the topology on $k[\gamma,\gamma^{-1}]$,
which is induced from the usual product topology on $\prod_{v\in\Sigma}k(\gamma)_{v}$,
is generated by those subset $\widetilde{Q}(\gamma)k[\gamma,\gamma^{-1}]$
with $\widetilde{Q}(T)\in k[T]$ not divisible by any element in the
set $\mathcal{S}$. Since $\mathcal{S}$ contains the polynomial $T$,
the second property implies that for each $i\in\{1,\ldots,M\}$ the
sequence $(\gamma^{e_{i,n}}-\gamma^{e'_{i,n}})_{n\in\mathcal{N}}$
converges to zero in $\prod_{v\in\Sigma}k(\gamma)_{v}$; this shows
that the two sequences $(\gamma^{e_{i,n}})_{n\in\mathcal{N}}$ and
$(\gamma^{e'_{i,n}})_{n\in\mathcal{N}}$ converge to the same element
in $\prod_{v\in\Sigma}k(\gamma)_{v}$. Hence, for each $i\in\{1,\ldots,M\}$,
the sequence $(\gamma^{e'_{i,n}})_{n\in\mathcal{N}}$ converges to
$x_{i}$ in $\prod_{v\in\Sigma}k(\gamma)_{v}$; since $x_{i}\in\prod_{v\in\Sigma}k(\gamma)_{v}^{*}$,
it follows from what is explained above that the same convergence
also happens in $\prod_{v\in\Sigma}k(\gamma)_{v}^{*}$. This shows
that $(x_{1},\ldots,x_{M})\in\overline{\{(\gamma^{e'_{1,n}},\ldots,\gamma^{e'_{M,n}})\}_{n\in\mathcal{N}}}\subset\overline{W(\Gamma)}$,
which completes the proof. 
\end{proof}

\begin{proof}[Proof of Theorem \ref{thm:main}]
Combine Proposition \ref{prop:red} and Theorem \ref{adelic}.
\end{proof}

\section{The Proof of Proposition \ref{key}\label{sec:The-Proof-of-key}}

The following result is proved in the author's recent work \cite{Sun-const_hyper}.
\begin{lem}
\label{lem:partition}Let $f(T)=\sum_{i\in I}c_{i}T^{e_{i}}\in k(T)$
with each $c_{i}\in k$ and \emph{$e_{i}\in\mathbb{Z}$}, where $I$
is a finite index set. Let $a\in\mathbb{N}$, $b\in\mathbb{N}\setminus p\mathbb{N}$
with $b$ greater than the cardinality of $I$. Denote by $\mathfrak{\mathscr{C}}$
the collection of those partitions $\mathscr{P}$ of the set $I$
such that for each set $\Omega\in\mathscr{P}$ we have $\sum_{i\in\Omega}c_{i}=0$
and for each nonempty proper subset $\Omega'\subset\Omega$ we have
$\sum_{i\in\Omega'}c_{i}\not=0$. Suppose that $g_{a,b}(T)$ divides
$f(T)$. Then there is some $\mathscr{P}\in\mathfrak{\mathscr{C}}$
such that for each set $\Omega\in\mathscr{P}$ and each $i_{1},i_{2}\in\Omega$
we have that $ab$ divides $e_{i_{1}}-e_{i_{2}}$. 
\end{lem}
Proved by an elementary linear-algebra argument, the following result
plays a crucial role so that Proposition 24 in the author's recent
work \cite{Sun-const_hyper} can be generalized to Proposition \ref{key},
which is the core in the proof of Theorem \ref{thm:main}.
\begin{lem}
\label{lem:lin_alg}Let $\mathcal{N}\subset\mathbb{N}$ be a subset
such that for each $m\in\mathbb{N}$ there is some $n\in\mathcal{N}$
divisible by $m$. Let $a_{j,i}\in\mathbb{Z}$ and $b_{j}\in\mathbb{Z}$,
$(j,i)\in\{1,\ldots,J\}\times\{1,\ldots,M\}$ be fixed integers. Suppose
that for each $n\in\mathcal{N}$ there are some $e_{i,n}$, $i\in\{1,\ldots,M\}$
such that $n$ divides $b_{j}-\sum_{i=1}^{M}a_{j,i}e_{i,n}$ for each
$j$. Then there is some $n_{0}\in\mathcal{N}$ with the following
property: for each $n\in\mathcal{N}$ divisible by $n_{0}$, there
are some $e'_{i,n}$, $i\in\{1,\ldots,M\}$, such that $\frac{n}{n_{0}}$
divides $e_{i,n}-e'_{i,n}$ and that $b_{j}=\sum_{i=1}^{M}a_{j,i}e'_{i,n}$
for each $j$.
\end{lem}
\begin{proof}
Consider the $J$-by-$(M+1)$ matrix $(a_{j,i}\,|\,b_{j})$, where
$j$ indices rows and $i$ indices the first $M$ columns. Applying
a sequence of the following operations: interchanging any two rows
or any two of the first $M$ columns, multiplying some row by an integer,
adding some row to another one, we can transform this matrix to $(a'_{j,i}\,|\,b'_{j})$
such that for some $R\le\min\{J,M\}$ we have that $a'_{j,i}=0$ for
any $(j,i)\in\left(\{1,\ldots,J\}\times\{1,\ldots,R\}\right)\cup\left(\{R+1,\ldots,J\}\times\{1,\ldots,M\}\right)$
with $i\neq j$, and that $a'_{i,i}\neq0$ if and only if $i\in\{1,\ldots,R\}$.
Then there is some permutation $\sigma$ on $\{1,\ldots,M\}$ such
that $n$ divides $b'_{j}-\sum_{i=1}^{M}a'_{j,i}e_{\sigma(i),n}$
for each $n\in\mathcal{N}$ and each $j\in\{1,\ldots,J\}$. By the
properties of $\mathcal{N}$, there is some $n_{0}\in\mathcal{N}$
divisible by $\prod_{i=1}^{R}a'_{i,i}$. For any $j\in\{1,\ldots,R\}$
and any $n\in\mathcal{N}$ divisible by $n_{0}$, from the fact that
\[
b'_{j}-\sum_{i=1}^{M}a'_{j,i}e_{\sigma(i),n}=b'_{j}-a'_{j,j}e_{\sigma(j),n}-\sum_{i=R+1}^{M}a'_{j,i}e_{\sigma(i),n}
\]
is divisible by $n\in a'_{j,j}\mathbb{Z}$, we see that $a'_{j,j}$
divides $b'_{j}-\sum_{i=R+1}^{M}a'_{j,i}e_{\sigma(i),n}$, and thus
there exists a unique $e'_{\sigma(j),n}\in\mathbb{Z}$ satisfying
$b'_{j}-a'_{j,j}e'_{\sigma(j),n}-\sum_{i=R+1}^{M}a'_{j,i}e_{\sigma(i),n}=0$;
hence $n$ divides $a'_{j,j}(e'_{\sigma(j),n}-e{}_{\sigma(j),n})$.
For any $j\in\{1,\ldots,R\}$, since $n_{0}$ divisible by $a'_{j,j}$,
we conclude that $e_{\sigma(j),n}-e'{}_{\sigma(j),n}$ is divisible
by $\frac{n}{a'_{j,j}}$ and thus by \textbf{$\frac{n}{n_{0}}$} as
desired. For any $j\in\{R+1,\ldots,J\}$ and any $n\in\mathcal{N}$
divisible by $n_{0}$, we simply define $e'_{\sigma(j),n}=e{}_{\sigma(j),n}$;
thus $\frac{n}{n_{0}}$ divides $e_{\sigma(j),n}-e'_{\sigma(j),n}$
trivially. For every pair $(j,i)\in\{R+1,\ldots,J\}\times\{1,\ldots,M\},$
we have $a'_{j,i}=0$, hence $b'_{j}$ is divisible by every integer,
and therefore $b'_{j}=0$. Combined with the construction of $e'_{\sigma(j),n}$
for any $j\in\{1,\ldots,R\}$, we see that for any $j\in\{1,\ldots,J\}$
and any $n\in\mathcal{N}$ divisible by $n_{0}$, we always have $b'_{j}-\sum_{i=1}^{M}a'_{j,i}e'_{\sigma(i),n}=0$.
Transforming the matrix $(a'_{j,i}\,|\,b'_{j})$ back to $(a_{j,i}\,|\,b_{j})$,
we obtain that $b_{j}-\sum_{i=1}^{M}a{}_{j,i}e'_{i,n}=0$ as desired.
\end{proof}
We are ready to present the
\begin{proof}[Proof of Proposition \ref{key}]
Choose $D\in\mathbb{N}$ such that for each $j\in\{1,\ldots,J\}$,
we may write 
\begin{eqnarray*}
f_{j}(X_{1},\ldots,X_{M}) & = & \sum_{(d_{0},d_{1},\cdots,d_{M})\in\{0,1,\cdots,D\}^{M+1}}c_{(j,d_{0},d_{1},\cdots,d_{M})}T^{d_{0}}X_{1}^{d_{1}}\cdots X_{M}^{d_{M}}
\end{eqnarray*}
with each $c_{(j,d_{0},d_{1},\cdots,d_{M})}\in k$. For each $j\in\{1,\ldots,J\}$,
denote by $\mathfrak{\mathscr{C}}_{j}$ the collection of those partitions
$\mathscr{P}$ of the set $\{0,1,\cdots,D\}^{M+1}$ such that for
each set $\Omega\in\mathscr{P}$ we have $\sum_{(d_{0},d_{1},\cdots,d_{M})\in\Omega}c_{(j,d_{0},d_{1},\cdots,d_{M})}=0$
and for each nonempty proper subset $\Omega'\subset\Omega$ we have
$\sum_{(d_{0},d_{1},\cdots,d_{M})\in\Omega'}c_{(d_{0},d_{1},\cdots,d_{M})}\not=0$.
By Remark 14 in \cite{Sun-const_hyper}, we may choose some $a_{0}\in\mathbb{N}\setminus p\mathbb{N}$
and $b_{0}\in\mathbb{N}\setminus p\mathbb{N}$ with $b_{0}>(D+1)^{M+1}$
such that for any $a\in a_{0}\mathbb{N}$ the polynomial $g_{a,b_{0}}(T)$
is not divisible by any element in $\mathcal{S}$. By our assumption,
there is a strictly increasing sequence $\{N_{a}\}_{a\in a_{0}\mathbb{N}}$
such that $g_{a,b_{0}}(T)$ divides 
\[
f_{j}(T^{e_{1,N_{a}}},\ldots,T^{e_{M,N_{a}}})=\sum_{(d_{0},d_{1},\cdots,d_{M})\in\{0,1,\cdots,D\}^{M+1}}c_{(j,d_{0},d_{1},\cdots,d_{M})}T^{d_{0}+d_{1}e{}_{1,N_{a}}+\cdots+d_{M}e{}_{M,N_{a}}}
\]
for any $j\in\{1,\ldots,J\}$. Thus, by Lemma \ref{lem:partition},
for any $a\in a_{0}\mathbb{N}$ and any $j\in\{1,\ldots,J\}$, there
is some $\mathscr{P}_{j,a}\in\mathfrak{\mathscr{C}}_{j}$ such that
for each set $\Omega\in\mathscr{P}_{j,a}$, each $(d_{0},d_{1},\cdots,d_{M})$
and $(d'_{0},d'_{1},\cdots,d'_{M})$ in $\Omega$, any $i\in\{1,\ldots,M\}$
and any $n\ge N_{a}$, we have that $ab_{0}$ divides both $d_{0}-d'_{0}+(d_{1}-d'_{1})e{}_{1,N_{a}}+\cdots+(d_{M}-d'_{M})e{}_{M,N_{a}}$.
Consider the subset $\{\prod_{i=1}^{a_{0}+n}i:\,\,n\in\mathbb{N}\}\subset a_{0}\mathbb{N}$.
For each $j\in\{1,\ldots,J\}$ the collection $\mathscr{C}_{j}$ is
finite while $\{\prod_{i=1}^{a_{0}+n}i:\,\,n\in\mathbb{N}\}$ is infinite;
thus there is an infinite subset $\mathcal{A}$ of the set $\{\prod_{i=1}^{a_{0}+n}i:\,\,n\in\mathbb{N}\}$,
which is contained in $a_{0}\mathbb{N}$, such that for each $j\in\{1,\ldots,J\}$
the collection $\{\mathfrak{\mathscr{P}}_{j,a}:a\in\mathcal{A}\}$
consists of only one partition, denoted by $\mathfrak{\mathscr{P}}_{j}$.
Since $\mathcal{A}\subset\{\prod_{i=1}^{a_{0}+n}i:\,\,n\in\mathbb{N}\}$
is an infinite subset, it has the property that for each $m\in\mathbb{N}$
there is some $a\in\mathcal{A}$ divisible by $m$.

For any $a\in\mathcal{A}$, any $j\in\{1,\ldots,J\}$, we observe
that $(e_{1,N_{a}},\cdots,e{}_{M,N_{a}})$ satisfies the condition
that for each set $\Omega\in\mathscr{P}_{j}$, each $(d_{0},d_{1},\cdots,d_{M})$
and $(d'_{0},d'_{1},\cdots,d'_{M})$ in $\Omega$, we have that $a$
divides $d_{0}-d'_{0}+(d_{1}-d'_{1})e{}_{1,N_{a}}+\cdots+(d_{M}-d'_{M})e{}_{M,N_{a}}$.
Applying Lemma \ref{lem:lin_alg}, we obtain some $n_{0}\in\mathcal{A}$
with the following property: for each $a\in\mathcal{A}$ divisible
by $n_{0}$, there are some $e'_{i,N_{a}}$, $i\in\{1,\ldots,M\}$,
such that $\frac{a}{n_{0}}$ divides $e_{i,N_{a}}-e'_{i,N_{a}}$ and
that for each $j\in\{1,\ldots,J\}$, each set $\Omega\in\mathscr{P}_{j}$,
each $(d_{0},d_{1},\cdots,d_{M})$ and $(d'_{0},d'_{1},\cdots,d'_{M})$
in $\Omega$, we have 
\[
d_{0}-d'_{0}+(d_{1}-d'_{1})e'{}_{1,N_{a}}+\cdots+(d_{M}-d'_{M})e'{}_{M,N_{a}}=0;
\]
thus we may let $m_{a,j,\Omega}=d_{0}+d_{1}e'{}_{1,N_{a}}+\cdots+d_{M}e'{}_{M,N_{a}}$
for any $(d_{0},d_{1},\cdots,d_{M})\in\Omega$. Letting $\mathcal{N}=\{N_{a}:\,a\in\mathcal{A}\cap n_{0}\mathbb{N}\}$,
which is an infinite subset of $\mathbb{N}$ since $\mathcal{A}\subset a_{0}\mathbb{N}$
and the sequence $\{N_{a}\}_{a\in a_{0}\mathbb{N}}$ is strictly increasing,
we now show that the constructed sequence $\{(e'_{1,n},\ldots,e'_{M,n})\}_{n\in\mathcal{N}}$
satisfies the desired properties. To verify Property (\ref{enu:p1}),
we fix some $j\in\{1,\ldots,J\}$ and $n=N_{a}\in\mathcal{N}$ with
$a\in\mathcal{A}\cap n_{0}\mathbb{N}$. From construction, we have
\[
\begin{array}{cl}
 & f_{j}(T^{e'_{1,n}},\ldots,T^{e'_{M,n}})\\
= & \sum_{(d_{0},d_{1},\cdots,d_{M})\in\{0,1,\cdots,D\}^{M+1}}c_{(j,d_{0},d_{1},\cdots,d_{M})}T^{d_{0}+d_{1}e'{}_{1,N_{a}}+\cdots+d_{M}e'{}_{M,N_{a}}}\\
= & \sum_{\Omega\in\mathscr{P}_{j}}T^{m_{a,j,\Omega}}\sum_{(d_{0},d_{1},\cdots,d_{M})\in\Omega}c_{(j,d_{0},d_{1},\cdots,d_{M})}\\
= & 0
\end{array}
\]
as desired. To verify Property (\ref{enu:p2}), we fix some $\widetilde{Q}(T)\in k[T]$,
not divisible by $T$. Since each zero of $\widetilde{Q}(T)$ is in
$(k^{\text{alg}})^{*}$ and thus has a finite order, we can use the
property that that for each $m\in\mathbb{N}$ there is some element
of $\mathcal{A}\cap n_{0}\mathbb{N}$ divisible by $m$ and get some
$a\in\mathcal{A}\cap n_{0}\mathbb{N}$ such that $\widetilde{Q}(T)$
divides $T^{a}-1$. Using this property again yields some $a_{\widetilde{Q}}\in\mathcal{A}\cap n_{0}\mathbb{N}$
divisible by $an_{0}$. Let $\widetilde{N}_{\widetilde{Q}}=N_{a_{\widetilde{Q}}}$%
{} and fix some $n\in\mathcal{N}$ with $n\ge\widetilde{N}_{\widetilde{Q}}=N_{a_{\widetilde{Q}}}$.
Then $n=N_{a'}\in\mathcal{N}$ with some $a'\in\mathcal{A}\cap n_{0}\mathbb{N}$;
the latter condition implies, by construction, that $\frac{a'}{n_{0}}$
divides $e_{i,N_{a'}}-e'_{i,N_{a'}}=e_{i,n}-e'_{i,n}$ for each $i\in\{1,\ldots,M\}$.
Since the sequence $\{N_{a}\}_{a\in a_{0}\mathbb{N}}$ is strictly
increasing, this implies that $a'\ge a_{\widetilde{Q}}$; by construction,
both $a'$ and $a_{\widetilde{Q}}$ are in the set $\{\prod_{i=1}^{a_{0}+n}i:\,\,n\in\mathbb{N}\}$,
thus we get that $a'$ is divisible by $a_{\widetilde{Q}}$. Because
$\frac{a_{\widetilde{Q}}}{n_{0}}$ is divisible by $a$, we conclude
that $a$ divides $\frac{a'}{n_{0}}$ and thus divides $e_{i,n}-e'_{i,n}$
for each $i\in\{1,\ldots,M\}$; equivalently, we have shown that $T^{e_{i,n}}-T^{e'_{i,n}}$
is divisible by $T^{a}-1$ and thus by $\widetilde{Q}(T)$ as desired.
This completes the proof.
\end{proof}

\section*{acknowledgement}

This research is supported by Ministry of Science and Technology.
The plan number is 104-2115-M-001-012-MY3. I appreciate the helpful
conversation with Fu-Tsun Wei during the construction of the example
in Remark \ref{rem:old_red}.

\bibliographystyle{alpha}
\bibliography{mybib}

\end{document}